\documentclass[12pt,reqno]{amsart}
\usepackage{amsmath, amsfonts, amssymb, amsthm}
\def\Z{\mathbb Z}

\def\N{\mathbb N}

\def\R{\mathbb R}

\def\fc{\mathfrak c}

\def\fs{{\mathfrak s}}
\def\fd{{\mathfrak d}}
\def\fe{{\mathfrak e}}
\def\fq{{\mathfrak q}}
\def\fX{{\mathfrak X}}

\def\1{{\bf 1}}

\def\pmod #1{\ ({\rm{mod}}\ #1)}
\def\mod #1{\ {\rm mod}\ #1}

\theoremstyle{plain}
\newtheorem{theorem}{Theorem}
\newtheorem{conjecture}{Conjecture}
\newtheorem{proposition}{Proposition}
\newtheorem{lemma}{Lemma}

\theoremstyle{definition}

\theoremstyle{remark}
\newtheorem*{Rem}{Remark}

\begin{document}
\title{Small gaps between the Piatetski-Shapiro primes}

\author{Hongze Li}
\address{Department of Mathematics, Shanghai Jiao Tong University, Shanghai 200240, People's Republic of China}
\email{lihz@sjtu.edu.cn}
\author{Hao Pan}
\address{Department of Mathematics, Nanjing University, Nanjing 210093,
People's Republic of China}
\email{haopan79@zoho.com}
\begin{abstract}
Suppose that $1<c<9/8$. For any $m\geq 1$, there exist infinitely many $n$ such that
$$
\{[n^c],\ [(n+1)^c],\ \ldots,\ [(n+k_0)^c]\}
$$
contains at least $m+1$ primes, if $k_0$ is sufficiently large (only depending on $m$ and $c$).
\end{abstract}
\maketitle

\section{Introduction}
\setcounter{lemma}{0}\setcounter{theorem}{0}\setcounter{proposition}{0}\setcounter{corollary}{0}
\setcounter{equation}{0}

Let $p_k$ denote the $k$-th prime. In view of the prime number theorem, the expected value of the prime gap $p_{k+1}-p_k$ is near to $\log p_k$. In 1940, Erd\H{o}s \cite{Er40} showed that
$$
\liminf_{k\to\infty}\frac{p_{k+1}-p_k}{\log p_k}\leq c_0
$$
for some  constant $0<c_0<1$. Later, the value of $c_0$ was successively improved. In 2009, using a refinement of the Selberg sieve method,  Goldston, Pintz and Y\i ld\i r\i m \cite{GPY09} proved that
$$
\liminf_{k\to\infty}\frac{p_{k+1}-p_k}{\log p_k}=0.
$$
Furthermore, under the Elliott-Halberstam conjecture, they also showed that
$$
\liminf_{k\to\infty}(p_{k+1}-p_k)\leq 16.
$$
In fact, if the twin prime conjecture is true, there are infinitely many $k$ such that $p_{k+1}-p_k=2$.
In 2014, Zhang \cite{Zh14} for the first time proved unconditionally that
$$
\liminf_{k\to\infty}(p_{k+1}-p_k)\leq 7\times 10^7,
$$
i.e.,
the gap $p_{k+1}-p_k$ can be infinitely often bounded by a finite number. One ingredient of Zhang's proof
is an improvement of the Bombieri-Vinogradov theorem for the smooth moduli. Subsequently, the bound for $p_{k+1}-p_k$ was rapidly reduced.
In 2015, with the help of a multi-dimensional sieve method, Maynard \cite{Ma15} gave a quite different proof of Zhang's result and improved the bound to $600$. Nowadays, the best known bound is $246$ \cite{Pol14}. Furthermore, using the Maynard's sieve method, Maynard and Tao independently found that
$$
\liminf_{k\to\infty}\frac{p_{k+m}-p_k}{\log p_k}\leq C_m
$$
for any $m\geq 1$, where $C_m$ is a positive constant only depending on $m$. Now, basing on the discussions of Maynard and Tao, the bounded gaps between the primes of some special forms are also investigated. For examples, the Maynard-Tao theorem has been extended to:
the primes $p$ with $p+2$ being an almost prime \cite{LP15}, the primes of the form $[\alpha n+\beta]$ \cite{CPS15}, the primes having a given primitive root (under GRH) \cite{Po14}, the primes $p=a^2+b^2$ with $a\leq \epsilon\sqrt{p}$ \cite{Th}, etc.. In fact, the Maynard-Tao theorem is valid for any subset of primes with positive relative upper density satisfying some mean value theorem, we refer the reader to \cite{Ma14}.

Let
$$
\N^c=\{[n^c]:\,n\in\N\},
$$
where $[x]=\max\{m\leq x:\,m\in\Z\}$.
That is, $\N^c$ the set of the integers of the form $[n^c]$.
In 1953, Piatetski-Shapiro \cite{PS53} showed that there are infinitely many primes lying in $\N^c$ provided $1<c<12/11$. In fact, he got
$$
\sum_{p\leq x,\,p\in \N^c}1=(1+o(1))\frac{x^{1/c}}{\log x}.
$$
The primes lying in $\N^c$ are usually called Piatetski-Shapiro primes, so that the Piatetski-Shapiro primes form a thin set of primes, and the average gap between the primes in $[1,x]\cap\N^c$ is about $x^{1-1/c}\log x$.
The upper bound for $c$ satisfying
$$
\sum_{p\leq x,\,p\in \N^c}1\gg\frac{x^{1/c}}{\log x}
$$
has been improved many times during the past six decades. Now the best admissible range of $c$ is $(1,243/205)$ by Rivat and Wu \cite{RW01}.

It is natural to ask how small the gaps between the primes in $\N^c$ can be.
Let $p_k^{(c)}$ be the $k$-th prime in $\N^c$.
Of course, we can't expect that the $p_{k+1}^{(c)}-p_k^{(c)}$ can be bounded by a finite number, since
$(n+1)^c-n^c\approx cn^{c-1}$
implies
$$
p_{k+1}^{(c)}-p_k^{(c)}\geq \big(c+o(1)\big)\cdot\big(p_k^{(c)}\big)^{1-1/c}.
$$
However, in this paper, we shall show that if $1<c<9/8$, then for any $m\geq 1$,
$$
\liminf_{k\to\infty}\frac{p_{k+m}^{(c)}-p_k^{(c)}}{\big(p_k^{(c)}\big)^{1-\gamma}}<C_m,
$$
where $\gamma=1/c$ and $C_m>0$ is a constant only depending on $m$ and $c$. That is, we have
\begin{equation}
p_{k+m}^{(c)}-p_k^{(c)}\leq C_m\big(p_k^{(c)}\big)^{1-\gamma}
\end{equation}
for infinitely many $k$. Our main result is that
\begin{theorem}\label{main}
Suppose that
$1<c<9/8$ and $m\geq 1$.
If
$$
k_0\geq e^{Cm}
$$
where $C>0$ is a constant only depending on $c$,
then there are infinitely many $n$ such that the set
$$
\{[n^c],\ [(n+1)^c],\ \ldots,\ [(n+k_0)^c]\}
$$
contains at least $m+1$ primes.
\end{theorem}
Although there also exists a mean value theorem for the Piatetski-Shapiro primes \cite{Pe03},
the main difficulty in the proof of Theorem \ref{main} is that the Piatetski-Shapiro primes are too sparse.
According to the Maynard sieve method, we have to use the weight
$$
\bigg(\sum_{\substack{d_i\mid [(n+i)^c]\\\text{for }0\leq i\leq k_0}}\lambda_{d_0,\ldots,d_k}\bigg)^2
$$
rather than
$$
\bigg(\sum_{\substack{d_i\mid n+h_i\\\text{for }0\leq i\leq k_0}}\lambda_{d_0,\ldots,d_k}\bigg)^2
$$
which is applicable to those subsets of primes with positive relative upper density.
Then our problem can be reduced to consider
$$
\sum_{\substack{X\leq n\leq 2X,\ n\equiv b\pmod{W}\\ d_i\mid [(n+i)^c]\text{ for }0\leq i\leq k_0}}\varpi([(n+i_0)^c])
$$
for some $0\leq i_0\leq k_0$, where $W$ is the product of the primes less than $\log\log\log X$ and $\varpi(n)=\log n$ or $0$ according to whether $n$ is prime or not. However, the above sum is not easy to estimate. Our strategy is to construct a suitable smooth function $\chi$ with $0\leq \chi(n)\leq 1$. Then it becomes possible to
estimate
$$
\sum_{\substack{X\leq n\leq 2X,\
n\equiv b\pmod{W}
\\ d_i\mid [(n+i)^c]\text{ for }0\leq i\leq k_0}}\chi([(n+i_0)^c])\cdot\varpi([(n+i_0)^c]).
$$
In fact, when $\chi([(n+i_0)^c])>0$, $[n^c],\ [(n+1)^c],\ \ldots,\ [(n+k_0)^c]$ form an arithmetic progression.
Furthermore, as we shall see later, we only need to use the Siegel-Walfisz theorem, rather than any mean value theorem.

In the next section, we shall extend the Maynard sieve method to the Piatetski-Shapiro primes. The proof of Theorem \ref{main} will be concluded in Section 3. Throughout this paper, $f(x)\ll g(x)$ means $f(x)=O\big(g(x)\big)$ as $x$ tends to $\infty$. And $\ll_\epsilon$ means the implied constant in $\ll$ only depends on $\epsilon$. Furthermore, as usual, we define $e(x)=\exp(2\pi\sqrt{-1}x)$ for $x\in\R$ and $\|x\|=\min_{m\in \mathbb{Z}}|x-m|$.

\section{Maynard's sieve method for the Piatetski-Shapiro primes}\label{sec2}
\setcounter{lemma}{0}\setcounter{theorem}{0}\setcounter{proposition}{0}\setcounter{corollary}{0}
\setcounter{equation}{0}\setcounter{proposition}{0}
Let $$\sigma_0=\frac{1}{200}\cdot\min\{c-1,\ 9-8c\},$$
and let ${k_0}$ be a large integer to be chosen later.
Suppose that
$f(t_0,t_1,\ldots,t_{k_0})$ is a symmetric smooth function whose support lies on the area
$$
\{(t_0,\ldots,t_{k_0}):\,t_0,\ldots,t_{k_0}\geq 0,\ t_0+\cdots+t_{k_0}\leq 1\}.
$$
Let $R=X^{\sigma_0}$ and define
$$
\lambda_{d_0,d_1,\ldots,d_{k_0}}=f\bigg(\frac{\log d_0}{\log R},\frac{\log d_1}{\log R},\ldots,\frac{\log d_{k_0}}{\log R}\bigg)\prod_{j=0}^{{k_0}}\mu(d_j).
$$
Clearly $\lambda_{d_0,\ldots,d_{k_0}}=0$ provided $d_0d_1\cdots d_{k_0}>R$.

Now suppose that $X$ is sufficiently large and let
$$
W=\prod_{p\leq\log\log\log X}p.
$$
For convenience, below we write $n\sim X$ for $X\leq n\leq 2X$.
The following lemma is one of the key ingredients of Maynard's sieve method.
\begin{lemma}\label{maynard}
\begin{align*}
&\sum_{\substack{d_0,\ldots,d_{k_0},e_0,\ldots,e_{k_0}\\
W,[d_0,e_0],\ldots,[d_{k_0},e_{k_0}]\text{ coprime}}}\frac{\lambda_{d_0,\ldots,d_{k_0}}\lambda_{e_0,\ldots,e_{k_0}}}{[d_0,e_0]\cdots[d_{k_0},e_{k_0}]}\\
=&\frac{1+o(1)}{(\log R)^{{k_0+1}}}\cdot\frac{W^{{k_0}+1}}{\phi(W)^{{k_0}+1}}\int_{\R^{{k_0}+1}}\bigg(\frac{\partial^{{k_0}+1}f(t_0,\ldots,t_{k_0})}{\partial t_0\cdots\partial t_{k_0}}\bigg)^{2}d t_0d t_1\cdots d t_{k_0}.
\end{align*}
\end{lemma}\begin{proof}
See \cite[Lemma 30]{Pol14}.
\end{proof}
Set
$$
\gamma=\frac1c.
$$
The conventional way to capture the Piatetski-Shapiro primes is to use the Fourier expansion of $\{x\}$ and the fact
$$
[(n+1)^\gamma]-[n^\gamma]=\begin{cases}1,&\text{if }n\in\N^c,\\
0&\text{otherwise},
\end{cases}
$$
when $\{n^\gamma\}>0$.
However, notice that $[(n+1)^\gamma]-[n^\gamma]=1$ if and only if $$
\{n^\gamma\}>1-((n+1)^\gamma-n^\gamma)=1-\gamma n^{\gamma-1}+O(n^{\gamma-2}).
$$
So if $\{n^\gamma\}$ lies in the short interval $[1-\delta n^{\gamma-1},1)$ for some constant $0<\delta<\gamma$, then $n\in\N^c$.
The following lemma is a classical result in number theory, and is frequently used for the problems of Diophantine approximation.
\begin{lemma}\label{smooth}
Suppose that $0\leq \alpha<\beta\leq1$ and $\Delta>0$ with $2\Delta<\beta-\alpha$. For any $r\geq 1$, there exists a smooth function $\psi$ with the period $1$ satisfying that\medskip

\noindent(i)
$\psi(x)=1$ if $\alpha+\Delta\leq\{x\}\leq\beta-\Delta$,
$\psi(x)=0$ if $\{x\}\leq\alpha$ or $\{x\}\geq\beta$, and $\psi(x)\in[0,1]$ otherwise;\medskip

\noindent(ii)
$$\psi(x)=(\beta-\alpha-\Delta)+\sum_{|j|\geq 1}a_je(jx),$$
where
$$
a_j\ll_r\min\bigg\{\frac1{|j|},\ \beta-\alpha-\Delta,\ \frac1{\Delta^{r}|j|^{r+1}}\bigg\}.
$$
\end{lemma}
\begin{proof}
This is \cite[Lemma 12 of Chapter I]{Vi54}.
\end{proof}
\begin{lemma}\label{bf92exp}
Let
$$
f_j(x)=jx^{\gamma}+C_1x+C_2x^{1-\gamma},
$$
where $C_1, C_2$ are constants and $|C_2|=o(X^{2\gamma-1})$. Suppose that $\sigma>0$ and $9(1-\gamma)+12\sigma<1$. Then there exists a sufficiently small $\epsilon>0$ (only depending on $c$ and $\sigma$),
such that
\begin{equation}
\min\bigg\{1,\frac{X^{1-\gamma}}{H}\bigg\}\cdot\sum_{j\sim H}\bigg|\sum_{n\sim X}\Lambda(n)e\big(f_j(n)\big)\bigg|\ll_\epsilon X^{1-\sigma-\epsilon},
\end{equation}
where
$1\leq H\leq X^{1-\gamma+\sigma+\epsilon}$.
\end{lemma}
\begin{proof}
This is just (2.10) of \cite{BF92}, although Balog and Friedlander only considered $f_j(x)=jx^{\gamma}+C_1x$. In fact, in their proof, only the fact
$
f_j''(x)\approx \gamma(\gamma-1)\cdot jx^{\gamma-2}
$
is used. So the same discussions are also valid for $f_j(x)=jx^{\gamma}+C_1x+C_2x^{1-\gamma}$.
\end{proof}
Suppose that $\epsilon_0$ is the $\epsilon$ corresponding to $c$ and $\sigma=8\sigma_0$ in Lemma \ref{bf92exp}.
Let
$$
\delta_0=\frac{c}{9},\qquad\eta_0=\frac{c\delta_0}{16{k_0}}.
$$
Suppose that
$\chi$ is the smooth function described in Lemma \ref{smooth} with
$$
\alpha=1-2\delta_0 X^{\gamma-1},\ \beta=1-\delta_0 X^{\gamma-1},\ \Delta=\frac{\beta-\alpha}{4},\ r=[100\epsilon_0^{-1}].
$$
And let $\psi$ be the smooth function described in Lemma \ref{smooth} with
$$
\alpha=\eta_0,\quad \beta=2\eta_0,\quad \Delta=\frac{\beta-\alpha}{4},\quad r=[100\sigma_0^{-1}].
$$
Define
$$
\varpi(n)=\begin{cases}\log n,&\text{if }n\text{ is prime},\\
0,&\text{otherwise}.
\end{cases}
$$
For $n\in\N^c$ and $h\in\Z$, let
$$\fs_{h}(n)=[([n^{\gamma}]+h+1)^c].
$$
\begin{proposition}\label{maynardshapiro}
\begin{align}
&\sum_{\substack{n\sim X,\ n\in\N^c\\ \fs_{j}(n)\equiv 1\pmod{W}\\\text{for }0\leq j\leq {k_0}}}
\bigg(\sum_{h=0}^{k_0}\varpi(\fs_{h}(n))\chi(\fs_h(n)^{\gamma})\psi(c\fs_h(n)^{1-\gamma})\bigg)\cdot
\bigg(\sum_{\substack{d_i\mid \fs_{i}(n)\\ 0\leq i\leq {k_0}}}\lambda_{d_0,d_1,\ldots,d_{k_0}}\bigg)^2\notag\\
=&\frac{1+o(1)}{(\log R)^{{k_0}}}\cdot\frac{W^{{k_0}-1}}{\phi(W)^{{k_0}+1}}\cdot\frac{9\delta_0\eta_0({k_0}+1)X^{\gamma}}{16}
\int_{\R^{{k_0}}}\bigg(\frac{\partial^{{k_0}}f(0,t_1,\ldots,t_{k_0})}{\partial t_1\cdots\partial t_{k_0}}\bigg)^{2}d t_1\cdots d t_{k_0}.
\end{align}
\end{proposition}
Assume that $n\sim X$ and
$$1-\frac{\delta_0}4 X^{\gamma-1}>\{n^{\gamma}\}> 1-4\delta_0 X^{\gamma-1}.$$
Then
$$
[(n+1)^{\gamma}]=[n^{\gamma}+\gamma n^{\gamma-1}+O(n^{\gamma-2})]>[n^{\gamma}].
$$
So $n\in\N^c$.
Furthermore, if $h$ is an integer with $|h|\leq {k_0}$,
we also have
\begin{align*}
&([n^{\gamma}]+h+1)^c=(n^{\gamma}-\{n^{\gamma}\}+h+1)^c\\
=&n+chn^{1-\gamma}+
c(1-\{n^{\gamma}\})n^{1-\gamma}
+O(n^{1-2\gamma})\\
=&n+h\cdot[cn^{1-\gamma}]+h\cdot\{cn^{1-\gamma}\}+
c(1-\{n^{\gamma}\})n^{1-\gamma}
+O(n^{1-2\gamma}).
\end{align*}
Clearly
$$
c(1-\{n^{\gamma}\})n^{1-\gamma}\leq 4c\delta_0 X^{\gamma-1}\cdot (2X)^{1-\gamma}\leq 2^{3-\gamma}c\delta_0
$$
and
$$
c(1-\{n^{\gamma}\})n^{1-\gamma}\geq \frac{c\delta_0}4 X^{\gamma-1} \cdot X^{1-\gamma}=\frac{c\delta_0}4.
$$
On the other hand, if
$$
\eta_0<\{cn^{1-\gamma}\}< 2\eta_0,
$$
then we have
$$
\{h\cdot cn^{1-\gamma}\}=h\cdot\{cn^{1-\gamma}\}\leq 2\eta_0{k_0}<\frac{c\delta_0}7
$$
for those $0\leq h\leq{k_0}$,
i.e.,
$$
\frac{c\delta_0}4\leq h\cdot \{cn^{1-\gamma}\}+c(1-\{n^{\gamma}\})n^{1-\gamma}\leq 1-\frac{\delta_0}2.
$$
Similarly, if $-{k_0}\leq h<0$, then
$$
-\eta_0\geq h\cdot\{cn^{1-\gamma}\}\geq -2\eta_0{k_0}\geq -\frac{c\delta_0}8
$$
and
$$
\{h\cdot cn^{1-\gamma}\}=1+h\cdot\{cn^{1-\gamma}\}\geq 1-\frac{c\delta_0}8.
$$
It follows that
$$
\frac{c\delta_0}9<h\cdot\{cn^{1-\gamma}\}+c(1-\{n^{\gamma}\})n^{1-\gamma}
\leq 2^{3-\gamma}c\delta_0<1
$$
and
$$
\{h\cdot cn^{1-\gamma}\}+c(1-\{n^{\gamma}\})n^{1-\gamma}\geq1+\frac{c\delta_0}8.
$$
Noting that
$$
(n^{\gamma}+h)^c=n+h\cdot cn^{1-\gamma}+O(n^{1-2\gamma}),
$$
we get
\begin{equation}\label{n1chc}
[([n^{\gamma}]+h+1)^c]=n+h\cdot[cn^{1-\gamma}]=\begin{cases}[(n^{\gamma}+h)^c],&\text{if }0\leq h\leq k_0,\\
[(n^{\gamma}+h)^c]+1,&\text{if }-k_0\leq h<0.\\
\end{cases}
\end{equation}
That is,
\begin{lemma}\label{nchi} If $n\sim X$,
$$1-\frac{\delta_0}4 X^{\gamma-1}>\{n^{\gamma}\}> 1-4\delta_0 X^{\gamma-1}$$
and $$\eta_0<\{cn^{1-\gamma}\}< 2\eta_0,$$
then for $h\in[-{k_0},{k_0}]$, we have
$$
\fs_{h}(n)=n+h\cdot[cn^{1-\gamma}].
$$
\end{lemma}
Assume that $n\sim X$ and $n\in\N^c$. Suppose that $0\leq h\leq{k_0}$ and let $m=\fs_h(n)$, i.e, $$m=[([n^{\gamma}]+h+1)^c].$$
Clearly
$$
m^{\gamma}-h-1< [n^{\gamma}]<(m+1)^{\gamma}-h-1=
m^{\gamma}-h-1+O(m^{\gamma-1}),
$$
when $\chi(m^{\gamma})>0$.
It follows that
\begin{equation}\label{ngmgh}
[n^{\gamma}]=[m^{\gamma}]-h.
\end{equation}
Assume that $\chi(m^{\gamma}),\ \psi(cm^{1-\gamma})>0$.
Then by Lemma \ref{nchi},
$$
[([n^{\gamma}]+h^*+1)^c]=[([m^{\gamma}]-h+h^*+1)^c]=m+(h^*-h)\cdot[cm^{1-\gamma}]
$$
for each $0\leq h^*\leq {k_0}$.

Furthermore, we must have $\{n^{\gamma}\}>0$. In fact, if $\{n^{\gamma}\}=0$, then
\begin{align}\label{ngamma0}
n=&([m^{\gamma}]-h)^c=m-(h+1)\cdot cm^{1-\gamma}+
c(1-\{m^{\gamma}\})m^{1-\gamma}
+O(m^{1-2\gamma})\notag\\
=&
m-(h+1)\cdot[cm^{1-\gamma}]+
c(1-\{m^{\gamma}\})m^{1-\gamma}-(h+1)\cdot\{cm^{1-\gamma}\}+O(m^{1-2\gamma}).
\end{align}
Note that $2\delta_0X^{\gamma-1}\geq 1-\{m^{\gamma}\}\geq \delta_0X^{\gamma-1}$ and $\{cm^{1-\gamma}\}\leq 2\eta_0$,
(\ref{ngamma0}) is impossible since
$$
2(k_0+1)\eta_0<c\delta_0/2<1/4.
$$
Since $n\in\N^c$, in view of (\ref{n1chc}) and (\ref{ngmgh}), we must have
$$
n=[([m^{\gamma}]-h+1)^{c}]=[(m^{\gamma}-h)^{c}]+1
$$
provided $h\geq 1$. Then
\begin{equation}\label{m1cn1c}
m^{\gamma}-h< n^{\gamma}\leq((m^{\gamma}-h)^{c}+1)^{\gamma}.
\end{equation}
But
\begin{align*}
((m^{\gamma}-h)^{c}+1)^{\gamma}=
(m^{\gamma}-h)+\gamma m^{\gamma-1}+O(m^{-1}),
\end{align*}
so in view of (\ref{ngmgh}), we must have
$$
\{n^{\gamma}\}>\{m^{\gamma}\}\geq 1-2\delta_0 X^{\gamma-1}.
$$
Moreover, we also have
$$
n^{1-\gamma}=(m-h\cdot[cm^{1-\gamma}])^{1-\gamma}=m^{1-\gamma}+O(m^{1-2\gamma}).
$$
Thus
\begin{lemma}\label{shnchi}
Suppose that $n\sim X$, $n\in\N^c$ and $1\leq h,h^*\leq{k_0}$. If
$$\chi(\fs_h(n)^{\gamma}),\ \psi(c\cdot\fs_h(n)^{1-\gamma})>0,$$ then
$$
\fs_{h^*}(n)=\fs_{h^*-h}(\fs_h(n)).
$$
We also have
$$
\{n^{\gamma}\}\geq1-2\delta_0 X^{\gamma-1}
$$
and
$$
\{cn^{1-\gamma}\}=\{c\cdot\fs_h(n)^{1-\gamma}\}+O(X^{1-2\gamma}).
$$
\end{lemma}
Conversely, assume that $0<h\leq {k_0}$ and $n=\fs_{-h}(m)$ with $\chi(m^\gamma),\psi(cm^{1-\gamma})>0$. 
Then
$$
n^{\gamma}\leq [m^\gamma]-h+1<(n+1)^\gamma.
$$
It is also impossible that $\{n^\gamma\}=0$, since if so, then
$$
n=([m^\gamma]-h+1)^c=m-h\cdot[cm^{1-\gamma}]+c(1-\{m^\gamma\})m^{1-\gamma}-h\cdot\{cm^{1-\gamma}\}+O(m^{1-2\gamma}),
$$
which will lead to a similar contradiction as (\ref{ngamma0}).
So we also have
$$
[m^\gamma]=[n^\gamma]+h,
$$
and
\begin{equation}\label{mngh}
m=[([m^\gamma]+1)^c]=[([n^\gamma]+h+1)^c]=\fs_{h}(n).
\end{equation}
According to Lemma \ref{shnchi} and (\ref{mngh}), for each $0\leq h\leq{k_0}$, we get
\begin{align}
&\sum_{\substack{n\sim X,\ n\in\N^c\\ \fs_{j}(n)\equiv 1\pmod{W}\\\text{for any }0\leq j\leq {k_0}}}
\varpi(\fs_{h}(n))\chi(\fs_h(n)^{\gamma})\psi(c\fs_h(n)^{1-\gamma})\bigg(\sum_{\substack{d_i\mid \fs_{i}(n)\\ 0\leq i\leq {k_0}}}\lambda_{d_0,\ldots,d_{k_0}}\bigg)^2\notag\\
=&\sum_{\substack{m\sim X\\ \fs_{j-h}(m)\equiv 1\pmod{W}\\ \text{for any }0\leq j\leq {k_0}}}\varpi(m)\chi(m^{\gamma})\psi(cm^{1-\gamma})\bigg(\sum_{\substack{d_i\mid \fs_{i-h}(m)\\ 0\leq i\leq {k_0}}}\lambda_{d_0,\ldots,d_{k_0}}\bigg)^2+O(\log X).
\end{align}

Below we just consider the case $h=0$, since all other cases are similar.
Clearly
\begin{align}\label{sumpichilambda}
&\sum_{\substack{n\sim X\\ \fs_{j}(n)\equiv 1\pmod{W}\\ \text{for any }0\leq j\leq {k_0}}}\varpi(n)\chi(n^{\gamma})\psi(cn^{1-\gamma})\bigg(\sum_{\substack{d_i\mid \fs_{i}(n)\\ 0\leq i\leq {k_0}}}\lambda_{d_0,\ldots,d_{k_0}}\bigg)^2\\
=&\sum_{\substack{d_1,\ldots,d_{k_0},e_1,\ldots,e_{k_0}
}}\lambda_{1,d_1,\ldots,d_{k_0}}\lambda_{1,e_1,\ldots,e_{k_0}}
\sum_{\substack{n\sim X\\ n,\fs_{1}(n),\ldots,\fs_{{k_0}}(n)\equiv 1\pmod{W}\\
[d_i,e_i]\mid \fs_{i}(n)\text{ for }
1\leq i\leq {k_0}}}\varpi(n)\chi(n^{\gamma})\psi(cn^{1-\gamma}).
\end{align}
Fix $d_1,\ldots,d_{k_0},e_1,\ldots,e_{k_0}$ with $d_1\cdots d_{k_0},e_1\cdots e_{k_0}\leq R$. We need to consider
\begin{align}\label{nfsnsum}
\sum_{\substack{n\sim X\\ n,\fs_{1}(n),\ldots,\fs_{{k_0}}(n)\equiv 1\pmod{W}\\
[d_i,e_i]\mid \fs_{i}(n)\text{ for }
1\leq i\leq {k_0}}}\varpi(n)\chi(n^{\gamma})\psi(cn^{1-\gamma}).
\end{align}

First, we claim that the sum (\ref{nfsnsum}) is $0$ unless those $[d_i,e_i]$ are pairwise coprime. In fact, assume that $[d_{i_1},e_{i_1}]$ and $[d_{i_2},e_{i_2}]$ have a common prime divisor $p$. Clearly $p\nmid W$, i.e., $p>{k_0}$. Recall that for each $1\leq i\leq {k_0}$,
$$
\fs_{i}(n)=n+i\cdot[cn^{1-\gamma}]
$$
provided $\chi(n^{\gamma}),\psi(cn^{1-\gamma})>0$. Thus we must have $p\mid [cn^{1-\gamma}]$ and $p\mid n$. It is impossible since $n$ is prime now.
Moreover, clearly $[d_i,e_i]$ is coprime to $W$ for each $1\leq i\leq{k_0}$.

Below we assume that $W,[d_1,e_1],\ldots,[d_{k_0},e_{{k_0}}]$ are pairwise coprime.
Clearly
\begin{align}\label{sumpichi}
&\sum_{\substack{n\sim X\\ n,\fs_{i}(n)\equiv 1\pmod{W}\\
[d_i,e_i]\mid \fs_{i}(n)\text{ for }
1\leq i\leq {k_0}}}\varpi(n)\chi(n^{\gamma})\psi(cn^{1-\gamma})
=\sum_{\substack{n\sim X\\ n\equiv 1\pmod{W}\\
[cn^{1-\gamma}]\equiv0\pmod{W}\\
n+i\cdot [cn^{1-\gamma}]\equiv0\pmod {[d_i,e_i]}\\
\text{for }1\leq i\leq {k_0}}}\varpi(n)\chi(n^{\gamma})\psi(cn^{1-\gamma})\notag\\
=&\sum_{\substack{n\sim X\\
n\equiv 1\pmod{W}}}\varpi(n)\chi(n^{\gamma})\psi(cn^{1-\gamma})\bigg(\frac1{W}\sum_{r_0=0}^{W-1}e
\bigg(\frac{r_0[cn^{1-\gamma}]}{W}\bigg)\bigg)\notag\\
&\cdot\prod_{\substack{1\leq i\leq {k_0}}}\bigg(\frac1{[d_i,e_i]}\sum_{r_i=0}^{[d_i,e_i]-1}e\bigg(\frac{r_i(n+i[cn^{1-\gamma}])}{[d_i,e_i]}\bigg)
\bigg)\notag\\
=&\frac{1}{W[d_1,e_1]\cdots[d_{k_0},e_{k_0}]}\sum_{\substack{0\leq r_0\leq W-1\\ 0\leq r_i\leq [d_i,e_i]-1}}\sum_{\substack{n\sim X\\
n\equiv 1\pmod{W}}}\varpi(n)\chi(n^{\gamma})\psi(cn^{1-\gamma})\notag\\
&\cdot e\bigg(n\sum_{i=1}^{k_0}\frac{r_i}{[d_i,e_i]}+[cn^{1-\gamma}]
\bigg(\frac{r_0}{W}+\sum_{i=1}^{k_0}\frac{r_ii}{[d_i,e_i]}\bigg)\bigg).
\end{align}
Fix $r_0,r_1,\ldots,r_{{k_0}}$ and let
$$
\theta_1=\sum_{i=1}^{k_0}\frac{r_i}{[d_i,e_i]},\qquad
\theta_2=\frac{r_0}{W}+\sum_{i=1}^{k_0}\frac{r_ii}{[d_i,e_i]}.
$$
\begin{lemma}\label{fit} For any $H\geq 2$,
\begin{align}\label{fint}
e(-\theta\{x\})=
\fc(\theta)\sum_{|h|\leq H}\frac{e(hx)}{h+\theta}+O(\Phi(x;H)\log H),
\end{align}
where $|\fc(\theta)|\leq\|\theta\|$ and $\Phi(x;H)=(1+H\|x\|)^{-1}$.
\end{lemma}
\begin{proof} (\ref{fint}) is an easy exercise for the Fourier series. We leave its proof to the reader.
\end{proof}
Let $H=X^{2\sigma_0}$. Note that $\Phi(cn^{1-\gamma};H)\ll H^{-1}$ if $\psi(cn^{1-\gamma})>0$. Then
\begin{align}\label{sumH}
&\sum_{\substack{n\sim X\\
n\equiv 1\pmod{W}}}\varpi(n)\chi(n^{\gamma})\psi(cn^{1-\gamma})e(n\theta_1+[cn^{1-\gamma}]\theta_2)\notag\\
=&\sum_{\substack{n\sim X\\
n\equiv 1\pmod{W}}}\varpi(n)\chi(n^{\gamma})\psi(cn^{1-\gamma})\cdot\bigg(\fc(\theta_2)\sum_{|h|\leq H}\frac{e\big(n\theta_1+(h+\theta_2)\cdot cn^{1-\gamma}\big)}{h+\theta_2}+O(H^{-1}\log H)\bigg).
\end{align}

Now using Lemma \ref{smooth} and letting $\alpha_0=3\delta_0 X^{\gamma-1}/4$, we get
\begin{align*}
\chi(n^{\gamma})
=&\alpha_0+\sum_{1\leq|j|\leq X^{1-\gamma+\epsilon_0}}\alpha_je(jn^{\gamma})+\sum_{|j|>X^{1-\gamma+\epsilon_0}}\alpha_je(jn^{\gamma})\\
=&\alpha_0+\sum_{1\leq|j|\leq X^{1-\gamma+\epsilon_0}}\alpha_je(jn^{\gamma})+O(X^{-2}),
\end{align*}
by noting that $$
\alpha_j\ll \bigg(\frac{1}{X^{\gamma-1}\cdot j}\bigg)^{10\epsilon_0^{-1}+2}\ll
\frac{X^{2-2\gamma}}{j^2}\cdot\bigg(\frac{1}{X^{\gamma-1}\cdot X^{1-\gamma+\epsilon_0}}\bigg)^{10\epsilon_0^{-1}}\ll
\frac{1}{X^2j^2}
$$
for those $j>X^{1-\gamma+\epsilon_0}$.
And we also have
\begin{align*}
\psi(cn^{1-\gamma})=\beta_0+\sum_{1\leq|j|\leq X^{\sigma_0}}\beta_je(jcn^{1-\gamma})+O(X^{-2}),
\end{align*}
where $\beta_0=3\eta_0/4$.
Thus
\begin{align}\label{sumchi12}
&\sum_{\substack{n\sim X\\
n\equiv 1\pmod{W}}}\varpi(n)\chi(n^{\gamma})\psi(cn^{1-\gamma})\cdot e\big(n\theta_1+(h+\theta_2)\cdot cn^{1-\gamma}\big)\notag\\
=&\sum_{\substack{|j_1|\leq X^{1-\gamma+\epsilon_0}\\ |j_2|\leq X^{\sigma_0}}}\alpha_{j_1}\beta_{j_2}\sum_{\substack{n\sim X\\
n\equiv 1\pmod{W}}}\varpi(n)e\big(n\theta_1+j_1n^{\gamma}+(j_2+h+\theta_2)\cdot cn^{1-\gamma}\big)+O(X^{-1}).
\end{align}

Recall that
$$\alpha_{j}\ll\min\{|j|^{-1},X^{\gamma-1}\}$$ in view of Lemma \ref{smooth}.
Applying the Lemma \ref{bf92exp}, for any given $j_2$, $h$ and $s$ with $|j_2|\leq  X^{\sigma_0}$ and $|h|\leq H$, we can get
\begin{align*}
&\sum_{\substack{1\leq |j_1|\leq X^{1-\gamma+\epsilon_0}}}|\alpha_{j_1}|\cdot\bigg|\sum_{\substack{n\sim X}}\varpi(n)e\big(n(\theta_1+sW^{-1})+j_1n^{\gamma}+(j_2+h+\theta_2)\cdot cn^{1-\gamma}\big)\bigg|\\
\ll& X^{\gamma-1}\log X\cdot\max_{1\leq Y\leq X^{1-\gamma+\epsilon_0}}\min\bigg\{1,\frac{X^{1-\gamma}}{Y}\bigg\}\\
&\cdot\sum_{\substack{j_1\sim Y}}\bigg|\sum_{\substack{n\sim X}}\varpi(n)e\big(n(\theta_1+sW^{-1})+j_1n^{\gamma}+(j_2+h+\theta_2)\cdot cn^{1-\gamma}\big)\bigg|\\
\ll&X^{\gamma-6\sigma_0}.
\end{align*}
That is,
\begin{align}\label{alphajbeta}
&\sum_{\substack{1\leq|j_1|\leq X^{1-\gamma+\epsilon_0}\\ |j_2|\leq X^{\sigma_0}}}\alpha_{j_1}\beta_{j_2}\sum_{\substack{n\sim X\\
n\equiv 1\pmod{W}}}\varpi(n)e\big(n\theta_1+j_1n^{\gamma}+(j_2+h+\theta_2)\cdot cn^{1-\gamma}\big)\notag\\
=&\frac{1}{W}\sum_{\substack{1\leq|j_1|\leq X^{1-\gamma+\epsilon_0}\\ |j_2|\leq X^{\sigma_0}}}\alpha_{j_1}\beta_{j_2}\sum_{n\sim X}\sum_{0\leq s\leq W-1}e(s(n-1)W^{-1})
\varpi(n)e\big(n\theta_1+j_1n^{\gamma}+(j_2+h+\theta_2)\cdot cn^{1-\gamma}\big)\notag\\
\ll&\frac{X^{\sigma_0}}{W}\sum_{\substack{1\leq|j_1|\leq X^{1-\gamma+\epsilon_0}\\ 0\leq s\leq W-1}}|\alpha_{j_1}|\cdot\bigg|\sum_{\substack{n\sim X}}\varpi(n)e\big(n(\theta_1+sW^{-1})+j_1n^{\gamma}+(j_2+h+\theta_2)\cdot cn^{1-\gamma}\big)\bigg|\notag\\
\ll&X^{\gamma-5\sigma_0}.
\end{align}

Below we need to show that
\begin{equation}\label{alpha0beta}
\alpha_0\sum_{\substack{|j_2|\leq X^{\sigma_0}}}\beta_{j_2}\sum_{\substack{n\sim X\\
n\equiv 1\pmod{W}}}\varpi(n)e\big(n\theta_1+(j_2+h+\theta_2)\cdot cn^{1-\gamma}\big)\ll \alpha_0 X^{1-5\sigma_0}
\end{equation}
unless $\theta_1=\theta_2=0$.
Assume that $r_0,r_1,\ldots,r_{k_0}$ are not all zero.
Since $W,[d_1,e_1],\cdots,[d_{k_0},e_{k_0}]$ are pairwise coprime and $d_1\cdots d_{k_0},e_1\cdots e_{k_0}\leq R=X^{\sigma_0}$, we must have
\begin{equation}\label{theta2lower}
\|\theta_2\|=\bigg\|\frac{r_0}{W}+\sum_{i=1}^{k_0}\frac{r_ii}{[d_i,e_i]}\bigg\|\geq \frac{1}{W\prod_{i=1}^{k_0}[d_i,e_i]}\geq\frac{1}{WX^{2\sigma_0}}.
\end{equation}

Suppose that $0\leq s\leq W-1$ and let
$$
f(x)=(\theta_1+sW^{-1})x+(j_2+h+\theta_2)\cdot cx^{1-\gamma}.
$$
In view of the Heath-Brown identity \cite{He83}, for any $\epsilon>0$,
\begin{equation}\label{heathbrown}
\sum_{\substack{n\sim X}}\varpi(n)e\big(f(n)\big)\ll_\epsilon X^\epsilon\max_{M}\bigg|\sum_{\substack{mn\sim X\\ m\sim M}}a_m b_ne\big(f(mn)\big)\bigg|,
\end{equation}
where $a_m,b_n$ satisfies one of the following three conditions:
\begin{equation}\label{sumtypeii}
|a_m|\leq1,\qquad |b_n|\leq1,
\end{equation}
\begin{equation}\label{sumtypeia}
|a_m|\leq1,\qquad b_n=1,
\end{equation}
\begin{equation}\label{sumtypeib}
|a_m|\leq1,\qquad b_n=\log n.
\end{equation}
Here the cases (\ref{sumtypeia}) and (\ref{sumtypeib}) are the Type I sums and
the case (\ref{sumtypeii}) is the Type II sum.
Furthermore, according to Proposition 1 of \cite{BF92}, we only need to consider the Type I sums for $M\leq X^{7/12+\sigma_0}$ and the Type II sums for $X^{2/3-3\sigma_0}\leq M\leq X^{5/6-\sigma_0}$.
We first consider the Type II sum. Clearly,
\begin{align*}
&\bigg|\sum_{\substack{mn\sim X\\ m\sim M}}a_m b_ne\big(f(mn)\big)\bigg|^2
\ll M\sum_{m\sim M}\bigg|\sum_{\substack{n\sim X/m}}b_ne\big(f(mn)\big)\bigg|^2\\
\ll&M\sum_{\substack{n_1,n_2\sim X/M}}\bigg|\sum_{m\sim M,X/n_1,X/n_2}e\big(f(mn_1)-f(mn_2)\big)\bigg|\\
\ll&XM+M\sum_{\substack{1\leq l\leq X/M\\ n,n+l\sim X/M}}\bigg|\sum_{m\sim M,X/(n+l),X/n}e\big(f(m(n+l))-f(mn)\big)\bigg|.\\
\end{align*}
\begin{lemma} Suppose that $\Delta>0$ and
$$
|f''(x)|\asymp\Delta
$$
for any $x\in[X,X+Y]$, where $f\asymp g$ means $f\ll g\ll f$.
Then
$$
\sum_{X\leq n\leq X+Y}e\big(f(n)\big)\ll Y\Delta^{1/2}+\Delta^{-1/2}.
$$
\end{lemma}
\begin{proof} This is the well-known van der Corput inequality (cf. \cite[Theorem 2.2]{GK91} ).
\end{proof}
Let $\theta_3=j_2+h+\theta_2$.
Note that
\begin{align*}
|f''((n+l)x)-f''(nx)|=&|\theta_3(n+l)^{1-\gamma}\cdot (\gamma-1)x^{-1-\gamma}-
\theta_3 n^{1-\gamma}\cdot (\gamma-1)x^{-1-\gamma}|\\
\asymp&\theta_3ln^{-\gamma}M^{-1-\gamma}\asymp \theta_3lX^{-\gamma}M^{-1}
\end{align*}
for any $x\sim M$ and $1\leq l\leq X/M$.
So
\begin{align*}
&\sum_{m\sim M,X/(n+l),X/n}e\big(f(m(n+l))-f(mn)\big)\\
\ll&M\cdot\theta_3^{1/2}l^{1/2}X^{-\gamma/2}M^{-1/2}+
\theta_3^{-1/2}l^{-1/2}X^{\gamma/2}M^{1/2}.
\end{align*}
Thus
\begin{align*}
&\bigg|\sum_{\substack{mn\sim X\\ m\sim M}}a_m b_ne\big(f(mn)\big)\bigg|^2\\
\ll&XM+M\cdot\frac{X}{M}\sum_{1\leq l\leq X/M}(\theta_3^{1/2}l^{1/2}X^{-\gamma/2}M^{1/2}+
\theta_3^{-1/2}l^{-1/2}X^{\gamma/2}M^{1/2})\\
\\
\ll&XM+
\theta_3^{1/2}X^{5/2-\gamma/2}M^{-1}+
\theta_3^{-1/2}X^{3/2+\gamma/2}\\
\ll&
XM+\theta_3^{1/2}X^{11/6+3\sigma_0-\gamma/2}+\theta_3^{-1/2}X^{3/2+\gamma/2}.
\end{align*}
Recalling $\theta_3=j_2+h+\theta_2$ and using (\ref{theta2lower}),
we get
$$
\frac{1}{WX^{2\sigma_0}}\leq|\theta_3|\leq 2|X^{2\sigma_0}|.
$$
It follows that
\begin{equation}
\bigg|\sum_{\substack{mn\sim X\\ m\sim M}}a_m b_ne\big(f(mn)\big)\bigg|^2\ll X^{11/6-\sigma_0}+X^{3/2+\gamma/2+2\sigma_0}\ll X^{2-14\sigma_0}.
\end{equation}

For the case (\ref{sumtypeib}), we also have
\begin{align*}
\bigg|\sum_{\substack{mn\sim X\\ m\sim M}}a_me\big(f(mn)\big)\log n\bigg|\leq&\sum_{\substack{m\sim M}}\bigg|\sum_{n\sim X/m}e\big(f(mn)\big)\int_1^n\frac{d t}t\bigg|\\
\leq&\int_1^{2X/M}\sum_{\substack{m\sim M}}\bigg|\sum_{\substack{n\sim X/m\\ n\geq t}}e\big(f(mn)\big)\bigg|\frac{d t}t.
\end{align*}
It suffices to consider those $t$ with $t\sim X/M$. Clearly
\begin{align*}
\sum_{\substack{m\sim M}}\bigg|\sum_{\substack{n\sim X/m\\ n\geq t}}e\big(f(mn)\big)\bigg|
=\sum_{\substack{m\sim M\\ m\leq 2X/t}}\bigg|\sum_{\substack{t\leq n\leq 2X/m}}e\big(f(mn)\big)\bigg|.
\end{align*}
Since
$$
|f''(mx)|\sim|\theta_3m^{1-\gamma}\cdot (\gamma-1)(X/M)^{-1-\gamma}|
$$
for $t\leq x\leq 2X/M$,
we have
$$
\sum_{\substack{t\leq n\leq 2X/M}}e\big(f(mn)\big)\ll \frac{X}{M}\cdot
\theta_3^{1/2}MX^{-1/2-\gamma/2}+\theta_3^{-1/2}M^{-1}X^{1/2+\gamma/2}.
$$
So
\begin{equation}
\sum_{\substack{mn\sim X\\ m\sim M}}a_me\big(f(mn)\big)\log n\ll
\theta_3^{1/2}MX^{1/2-\gamma/2}+\theta_3^{-1/2}X^{1/2+\gamma/2}\ll X^{1-7\sigma_0}.
\end{equation}
Similarly, it is not difficult to show that
\begin{equation}
\sum_{\substack{mn\sim X\\ m\sim M}}a_me\big(f(mn)\big)
\ll X^{1-7\sigma_0}.
\end{equation}
Thus noting that $|j_2|\leq X^{\sigma_0}$ and using (\ref{heathbrown}), we get (\ref{alpha0beta}).

Now in view of (\ref{sumH}), (\ref{sumchi12}), (\ref{alphajbeta}) and (\ref{alpha0beta}), we have
\begin{align*}
&\sum_{\substack{n\sim X\\
n\equiv 1\pmod{W}}}\varpi(n)\chi(n^{\gamma})\psi(cn^{1-\gamma})e(n\theta_1+[cn^{1-\gamma}]\theta_2)\notag\\
\ll&\alpha_0 X^{1-5\sigma_0}\log X+\frac1H\sum_{\substack{n\sim X\\
n\equiv 1\pmod{W}}}\varpi(n)\chi(n^{\gamma})\psi(cn^{1-\gamma}),
\end{align*}
provided that $\theta_1\neq0$ or $\theta_2\neq0$.
Using the similar discussions, we can also get
\begin{align*}
&\sum_{\substack{n\sim X\\
n\equiv 1\pmod{W}}}\varpi(n)\chi(n^{\gamma})\psi(cn^{1-\gamma})\\
=&\sum_{\substack{|j_1|\leq X^{1-\gamma+\epsilon_0}\\ |j_2|\leq X^{\sigma_0}}}\alpha_{j_1}\beta_{j_2}
\sum_{\substack{n\sim X\\
n\equiv 1\pmod{W}}}\varpi(n)e(j_1n^\gamma+j_2\cdot cn^{1-\gamma})+O(X^{-1})\\
=&\alpha_0\beta_0\sum_{\substack{n\sim X\\
n\equiv 1\pmod{W}}}\varpi(n)+O(X^{\gamma-4\sigma_0}).
\end{align*}
It follows that
\begin{align*}
&\sum_{\substack{n\sim X\\
n\equiv 1\pmod{W}}}\varpi(n)\chi(n^{\gamma})\psi(cn^{1-\gamma})e(n\theta_1+[cn^{1-\gamma}]\theta_2)\\
=&\begin{cases}
\alpha_0\beta_0\sum_{\substack{n\sim X\\
W\mid n-1}}\varpi(n)+O(X^{\gamma-4\sigma_0}),&\text{if }\theta_1=\theta_2=0,\\
O(X^{\gamma-4\sigma_0}),&\text{otherwise}.
\end{cases}
\end{align*}
That is, in view of (\ref{sumpichi}),
\begin{align*}
&\sum_{\substack{n\sim X\\ n,\fs_{1}(n),\ldots,\fs_{{k_0}}(n)\equiv 1\pmod{W}\\
[d_i,e_i]\mid \fs_{i}(n)\text{ for }
1\leq i\leq {k_0}}}\varpi(n)\chi(n^{\gamma})\psi(cn^{1-\gamma})\\
=&\frac{1}{W[d_1,e_1]\cdots[d_{k_0},e_{k_0}]}\cdot\alpha_0\beta_0\sum_{\substack{n\sim X\\
n\equiv 1\pmod{W}}}\varpi(n)+O(X^{\gamma-3\sigma_0}).
\end{align*}
It follows from (\ref{sumpichilambda}) that
\begin{align*}
&\sum_{\substack{n\sim X\\ \fs_{j}(n)\equiv 1\pmod{W}\\ \text{for }0\leq j\leq {k_0}}}\varpi(n)\chi(n^{\gamma})\psi(cn^{1-\gamma})\bigg(\sum_{\substack{d_i\mid \fs_{i}(n)\\ 0\leq i\leq {k_0}}}\lambda_{d_0,d_1,\ldots,d_{k_0}}\bigg)^2\\
=&\alpha_0\beta_0\sum_{\substack{d_1,\ldots,d_{k_0},e_1,\ldots,e_{k_0}\\
d_1\cdots d_{k_0},e_1\cdots e_{k_0}\leq R\\
W,[d_1,e_1,],\ldots,[d_{k_0},e_{k_0}]\text{ coprime}}}\frac{\lambda_{1,d_1,\ldots,d_{k_0}}\lambda_{1,e_1,\ldots,e_{k_0}}}{W[d_1,e_1]\cdots[d_{k_0},e_{k_0}]}
\cdot\sum_{\substack{n\sim X\\
n\equiv 1\pmod{W}}}\varpi(n)+O(X^{\gamma-\sigma_0}).
\end{align*}

By the Siegel-Walfisz theorem,
$$
\sum_{\substack{n\sim X\\
n\equiv 1\pmod{W}}}\varpi(n)=\frac{(1+o(1))X}{\phi(W)}.
$$
And by Lemma \ref{maynard}, we have
\begin{align*}
&\sum_{\substack{d_1,\ldots,d_{k_0},e_1,\ldots,e_{k_0}\\
d_1\cdots d_{k_0},e_1\cdots e_{k_0}\leq R\\
W,[d_1,e_1,],\ldots,[d_{k_0},e_{k_0}]\text{ co-prime}}}\frac{\lambda_{1,d_1,\ldots,d_{k_0}}\lambda_{1,e_1,\ldots,e_{k_0}}}{W[d_1,e_1]\cdots[d_{k_0},e_{k_0}]}\\
=&\frac{1+o(1)}{(\log R)^{{k_0}}}\cdot\frac{W^{{k_0}-1}}{\phi(W)^{{k_0}}}\int_{\R^{{k_0}}}\bigg(\frac{\partial^{{k_0}}f(0,t_1,\ldots,t_{k_0})}{\partial t_1\cdots\partial t_{k_0}}\bigg)^{2}d t_1\cdots d t_{k_0}.
\end{align*}
Finally, since $f(t_0,\ldots,t_{k_0})$ is symmetric, Proposition \ref{maynardshapiro} is concluded.\qed

\section{The Proof of Theorem \ref{main}}
\setcounter{lemma}{0}\setcounter{theorem}{0}\setcounter{corollary}{0}
\setcounter{equation}{0}\setcounter{proposition}{0}

Let
$\chi^\circ$ be the smooth function described in Lemma \ref{smooth} with
$$
\alpha=1-4\delta_0 X^{\gamma-1},\ \beta=1,\ \Delta=\frac{\beta-\alpha}{4},\ r=[100\epsilon_0^{-1}],
$$
and let $\chi^*(t)=\chi^\circ(t-\delta_0 X^{\gamma-1})$.
Clearly
$\chi^*(n^{\gamma})=1$ if $\{n^{\gamma}\}\geq 1-2\delta_0 X^{\gamma-1}$. That is, in view of Lemma \ref{shnchi}, we have
\begin{equation}\label{chi12to3}
\chi(\fs_h(n)^{\gamma}),\ \psi(c\cdot\fs_h(n)^{1-\gamma})>0\ \Longrightarrow\
\chi^*(n^{\gamma})=1
\end{equation}
for any $0\leq h\leq k_0$.
Let $\psi^*$ be the smooth function described in Lemma \ref{smooth} with
$$
\alpha=\frac{\eta_0}{2},\quad \beta=\frac{5\eta_0}{2},\quad \Delta=\frac{\beta-\alpha}{4},\quad r=[100\sigma_0^{-1}].
$$
By Lemma \ref{shnchi}, we also have
\begin{equation}\label{chi2to4}
\psi(c\cdot\fs_h(n)^{1-\gamma})>0\ \Longrightarrow\
\psi^*(c\cdot n^{1-\gamma})=1.
\end{equation}
for any $0\leq h\leq k_0$.

Furthermore,
Assume that $\chi^*(n^{\gamma}),\psi^*(cn^{1-\gamma})>0$. Then
$$
\{n^{\gamma}\}\in[1-3\delta_0X^{\gamma-1},1)\cup[0,\delta_0X^{\gamma-1}],
$$
if $n\in \mathbb{N}^c$, we have
$$
\{n^{\gamma}\}\in[1-3\delta_0X^{\gamma-1},1),
$$
by the discussion of Lemma \ref{nchi}, we know that for each $0\leq h\leq{k_0}$,
$$
\fs_h(n)=n+h\cdot[cn^{1-\gamma}],
$$
when $\chi^*(n^{\gamma}),\psi^*(cn^{1-\gamma})>0$ and $n\in \mathbb{N}^c$.
Thus
\begin{align}\label{sumchi34nc}
&\sum_{\substack{n\sim X,\ n\in\N^c\\ \fs_j(n)\equiv1\pmod{W}\\
\text{for }0\leq j\leq{k_0}}}\chi^*(n^{\gamma})
\psi^*(cn^{1-\gamma})\bigg(\sum_{\substack{d_i\mid \fs_i(n)\\ 0\leq i\leq{k_0}}}\lambda_{d_0,\ldots,d_{k_0}}\bigg)^2\notag\\
=&\sum_{\substack{n\sim X,\ n\in\N^c\\ n+j\cdot[cn^{1-\gamma}]\equiv1\pmod{W}\\
\text{for }0\leq j\leq{k_0}}}\chi^*(n^{\gamma})
\psi^*(cn^{1-\gamma})\bigg(\sum_{\substack{d_i\mid n+i\cdot[cn^{1-\gamma}]\\ 0\leq i\leq{k_0}}}\lambda_{d_0,\ldots,d_{k_0}}\bigg)^2\notag\\
\leq&\sum_{\substack{n\sim X\\ n+j\cdot[cn^{1-\gamma}]\equiv1\pmod{W}\\
\text{for }0\leq j\leq{k_0}}}\chi^*(n^{\gamma})
\psi^*(cn^{1-\gamma})\bigg(\sum_{\substack{d_i\mid n+i\cdot[cn^{1-\gamma}]\\ 0\leq i\leq{k_0}}}\lambda_{d_0,\ldots,d_{k_0}}\bigg)^2.
\end{align}

Let $\fX_q$ be the set of those $(d_0,\ldots,d_{k_0},e_0,\ldots,e_{k_0})$ with $d_0\cdots d_{k_0},e_0\cdots e_{k_0}\leq R$ satisfying

\medskip
\noindent(i) $d_0,\ldots,d_{k_0},e_0,\ldots,e_{k_0}$ are square-free;

\medskip
\noindent(ii) $d_0,\ldots,d_{k_0},e_0,\ldots,e_{k_0}$ are coprime to $W$;

\medskip
\noindent(iii) $q$ is the product of those primes $p$ such that $p$ divides $([d_i,e_i],[d_j,e_j])$ for some $0\leq i<j\leq k_0$;\medskip

\noindent 
Clearly for any square-free $d_0,\ldots,d_{k_0},e_0,\ldots,e_{k_0}$ with $d_0\cdots d_{k_0},e_0\cdots e_{k_0}\leq R$ and $(d_je_j,W)=1$, there exists unique $q\leq R^2$ such that $(d_0,\ldots,d_{k_0},e_0,\ldots,e_{k_0})\in\fX_q$.
Note that if $p>k_0$ divides both $n+i\cdot[cn^{1-\gamma}]$ and $n+j\cdot[cn^{1-\gamma}]$ for $0\leq i<j\leq k_0$, then $p$ divides both $n$ and $[cn^{1-\gamma}]$. So we have
\begin{align}\label{sumchi34}
&\sum_{\substack{n\sim X\\ n+j\cdot[cn^{1-\gamma}]\equiv1\pmod{W}\\
\text{for }0\leq j\leq{k_0}}}\chi^*(n^{\gamma})
\psi^*(cn^{1-\gamma})\bigg(\sum_{\substack{d_i\mid n+i\cdot[cn^{1-\gamma}]\\ 0\leq i\leq{k_0}}}\lambda_{d_0,\ldots,d_{k_0}}\bigg)^2\notag\\
=&\sum_{q\leq R^2}\sum_{(d_0,\ldots,d_{k_0},e_0,\ldots,e_{k_0})\in\fX_q}\lambda_{d_0,\ldots,d_{k_0}}\lambda_{e_0,\ldots,e_{k_0}}\sum_{\substack{n\sim X\\
n+j\cdot[cn^{1-\gamma}]\equiv1\pmod{W}\\
n+j\cdot[cn^{1-\gamma}]\equiv0\pmod{[d_j,e_j]}\\
\text{for each }0\leq j\leq{k_0}
}}\chi^*(n^{\gamma})
\psi^*(cn^{1-\gamma})\notag\\
=&\sum_{q\leq R^2}\sum_{(d_0,\ldots,d_{k_0},e_0,\ldots,e_{k_0})\in\fX_q}\lambda_{d_0,\ldots,d_{k_0}}\lambda_{e_0,\ldots,e_{k_0}}\sum_{\substack{n\sim X\\
n,[cn^{1-\gamma}]\equiv0\pmod{q}\\ n-1,[cn^{1-\gamma}]\equiv0\pmod{W}\\
n+j\cdot[cn^{1-\gamma}]\equiv0\pmod{\fq_j^{-1}[d_j,e_j]}\\
\text{for each }0\leq j\leq{k_0}
}}\chi^*(n^{\gamma})
\psi^*(cn^{1-\gamma}),
\end{align}
where $\fq_j=(q,[d_j,e_j])$. Clearly $W,q,\fq_0^{-1}[d_0,e_0],\ldots,\fq_{k_0}^{-1}[d_{k_0},e_{k_0}]$ are co-prime to each other, according to the definition of $\fX_q$.
Since $d_0\cdots d_{k_0},e_0\cdots e_{k_0}\leq R$, using the similar discussions in Section \ref{sec2}, we can prove without any difficulty that
\begin{equation}\label{nfschi34}
\sum_{\substack{n\sim X\\
n,[cn^{1-\gamma}]\equiv0\pmod{q}\\ n-1,[cn^{1-\gamma}]\equiv0\pmod{W}\\
n+j\cdot[cn^{1-\gamma}]\equiv0\pmod{\fq_j^{-1}[d_j,e_j]}\\
\text{for each }0\leq j\leq{k_0}
}}\chi^*(n^{\gamma})
\psi^*(cn^{1-\gamma})=\frac{9\delta\eta X^{\gamma}}{2W^2q^2\prod_{j=0}^{k_0}(\fq_j^{-1}[d_j,e_j])}+O(X^{\gamma-6\sigma_0}).
\end{equation}
Let $\fd_j=(q,d_j)$ and $\fe_j=(q,e_j)$ for each $0\leq j\leq{k_0}$. And let $d_j^*=d_j/\fd_j$ and $e_j^*=e_j/\fe_j$.
Clearly
\begin{align*}
\lambda_{d_0,\ldots,d_{k_0}}=&f\bigg(\frac{\log d_0}{\log R},\ldots,\frac{\log d_{k_0}}{\log R}\bigg)\prod_{j=0}^{k_0}\mu(d_j)\\
=&\prod_{j=0}^{k_0}\mu(\fd_j)\cdot f\bigg(\frac{\log \fd_0+\log d_0^*}{\log R},\ldots,\frac{\log \fd_{k_0}+\log d_{k_0}^*}{\log R}\bigg)\prod_{j=0}^{k_0}\mu(d_j^*).
\end{align*}
Let
$$
f_{\fd_0,\ldots,\fd_{k_0}}(t_0,\ldots,t_{k_0})=f\bigg(t_0+\frac{\log \fd_0}{\log R},\ldots,t_{k_0}+\frac{\log \fd_{k_0}}{\log R}\bigg).
$$
Then by Lemma \ref{maynard}, for any fixed $\fd_0,\fe_0,\ldots,\fd_{k_0},\fe_{k_0}$, we have
\begin{align*}
&\sum_{\substack{(d_0,\ldots,d_{k_0},e_0,\ldots,e_{k_0})\in\fX_q\\
(q,d_j)=\fd_j,\ (q,e_j)=\fe_j\\
\text{for each }0\leq j\leq{k_0}}}\frac{\lambda_{d_0,\ldots,d_{k_0}}\lambda_{e_0,\ldots,e_{k_0}}}{\prod_{j=0}^{k_0}(\fq_j^{-1}[d_j,e_j])}
=\sum_{\substack{d_0^*,\ldots,d_{k_0}^*,e_0^*,\ldots,e_{k_0}^*\\
(qW,[d_j^*,e_j^*])=1\\
\text{for each }0\leq j\leq{k_0}}}\frac{\lambda_{\fd_0 d_0^*,\ldots,\fd_{k_0} d_{k_0}^*}\lambda_{\fe_0 e_0^*,\ldots,\fe_{k_0} e_{k_0}^*}}{\prod_{j=0}^{k_0}[d_j^*,e_j^*]}\\
=&\frac{1}{(\log R)^{{k_0}+1}}\cdot\frac{(qW)^{{k_0}+1}}{\phi(qW)^{{k_0}+1}}\int_{\R^{{k_0}+1}}\bigg(\frac{\partial^{{k_0}+1} f_{\fd_0,\ldots,\fd_{k_0}}(t_0,\ldots,t_{k_0})}{\partial t_0\cdots\partial t_{k_0}}\bigg)^2 d t_0\cdots d t_{k_0}.
\end{align*}
Assume that
$$
M=\max_{(t_0,\ldots,t_{k_0})\in\R^{{k_0}+1}}\bigg|\frac{\partial^{{k_0}+1} f(t_0,\ldots,t_{k_0})}{\partial t_0\cdots\partial t_{k_0}}\bigg|.
$$
Since $q$ is coprime to $W$, $q>1$ implies $q>\log\log\log X$.
Then
\begin{align*}
&\sum_{\substack{2\leq q\leq R^2\\
(d_0,\ldots,d_{k_0},e_0,\ldots,e_{k_0})\in\fX_q}}\frac{\lambda_{d_0,\ldots,d_{k_0}}
\lambda_{e_0,\ldots,e_{k_0}}}{W^2q^2\prod_{j=0}^{k_0}(\fq_j^{-1}[d_j,e_j])}\\
\ll&\sum_{q\geq \log\log\log\log X}\frac{1}{q^{2-\epsilon}}\cdot\frac{1}{(\log X^{\delta_0})^{{k_0}+1}}\cdot\frac{W^{{k_0}-1}M^2}{\phi(W)^{{k_0}+1}}\\
=&o\bigg(\frac{1}{(\log R)^{{k_0}+1}}\cdot\frac{W^{{k_0}-1}}{\phi(W)^{{k_0}+1}}\bigg).
\end{align*}
Thus in view of (\ref{sumchi34}) and (\ref{nfschi34}), we have
\begin{align*}
&\sum_{\substack{n\sim X\\ n+j\cdot[cn^{1-\gamma}]\equiv1\pmod{W}\\
\text{for each }0\leq j\leq{k_0}}}\chi^*(n^{\gamma})
\psi^*(cn^{1-\gamma})\bigg(\sum_{\substack{d_i\mid n+i\cdot[cn^{1-\gamma}]\\ 0\leq i\leq{k_0}}}\lambda_{d_0,\ldots,d_{k_0}}\bigg)^2\notag\\
=&\frac{9\delta\eta X^{\gamma}}{2}\sum_{\substack{d_0,\ldots,d_{k_0},e_0,\ldots,e_{k_0}\\
W,[d_0,e_0],\ldots,[d_{k_0},e_{k_0}]
\text{ coprime}}}\frac{\lambda_{d_0,\ldots,d_{k_0}}\lambda_{e_0,\ldots,e_{k_0}}}
{W^2\prod_{j=0}^{k_0}[d_j,e_j]}+o\bigg(\frac{X^{\gamma}}{(\log R)^{{k_0}+1}}\cdot\frac{W^{{k_0}-1}}{\phi(W)^{{k_0}+1}}\bigg)\notag\\
=&(1+o(1))\cdot \frac{9\delta\eta X^{\gamma}}{2(\log R)^{{k_0}+1}}\cdot\frac{W^{{k_0}-1}}{\phi(W)^{{k_0}+1}}\int_{\R^{{k_0}+1}}\bigg(\frac{\partial^{{k_0}+1} f(t_0,\ldots,t_{k_0})}{\partial t_0\cdots\partial t_{k_0}}\bigg)^2 d t_0\cdots d t_{k_0}.
\end{align*}
Combining the above equation with (\ref{sumchi34nc}), we get
\begin{proposition}\label{maynardnc}
\begin{align}
&\sum_{\substack{n\sim X,\ n\in\N^c\\ \fs_j(n)\equiv1\pmod{W}\\
\text{for each }0\leq j\leq{k_0}}}\chi^*(n^{\gamma})
\psi^*(cn^{1-\gamma})\bigg(\sum_{\substack{d_i\mid \fs_i(n)\\ 0\leq i\leq{k_0}}}\lambda_{d_0,\ldots,d_{k_0}}\bigg)^2\notag\\
\leq&\frac{1+o(1)}{(\log R)^{{k_0}+1}}\cdot\frac{W^{{k_0}-1}}{\phi(W)^{{k_0}+1}}\cdot \frac{9\delta_0\eta_0 X^{\gamma}}{2}\int_{\R^{{k_0}+1}}\bigg(\frac{\partial^{{k_0}+1} f(t_0,\ldots,t_{k_0})}{\partial t_0\cdots\partial t_{k_0}}\bigg)^2 d t_0\cdots d t_{k_0}.
\end{align}
\end{proposition}
Now we are ready to prove Theorem \ref{main}.
According to the discussions of \cite[Sections 6.8-6.9]{Ga15},
we may construct a symmetric smooth function $f(t_0,t_1,\ldots,t_{k_0})$ such that
$$
\frac{(k_0+1)\int_{\R^{{k_0}}}\big(\frac{\partial^{{k_0}} f(t_0,t_1,\ldots,t_{k_0})}{\partial t_1\cdots\partial t_{k_0}}\big)^2 d t_1\cdots d t_{k_0}}{\int_{\R^{{k_0}+1}}\big(\frac{\partial^{{k_0}+1} f(t_0,\ldots,t_{k_0})}{\partial t_0\cdots\partial t_{k_0}}\big)^2 d t_0\cdots d t_{k_0}}\geq
\frac12\log k_0+\frac12\log\log k_0-2.
$$
So when
$$
\frac12\log k_0+\frac12\log\log k_0-2>8m\sigma_0^{-1},
$$
in view of Propositions \ref{maynardshapiro} and \ref{maynardnc}, we get that the sum
\begin{align*}
&\sum_{\substack{n\sim X,\ n\in\N^c\\ \fs_{j}(n)\equiv 1(\mod W)\\\text{for }0\leq j\leq {k_0}}}
w_n\bigg(\sum_{h=0}^{k_0}\varpi(\fs_{h}(n))\chi(\fs_h^{\gamma}(n))\psi(c\fs_h(n)^{1-\gamma})-m\chi^*(n^{\gamma})
\psi^*(cn^{1-\gamma})\log 3X\bigg)
\end{align*}
is positive,
where
$$
w_n=\bigg(\sum_{\substack{d_i\mid \fs_{i}(n)\\ 0\leq i\leq {k_0}}}\lambda_{d_0,d_1,\ldots,d_{k_0}}\bigg)^2\geq0.
$$
It follows that there exists $n\in[X,2X]\cap\N^c$  satisfying
$$
\sum_{h=0}^{k_0}\varpi(\fs_{h}(n))\chi(\fs_h^{\gamma}(n))\psi(c\fs_h(n)^{1-\gamma})-m\chi^*(n^{\gamma})
\psi^*(cn^{1-\gamma})\log 3X>0.
$$
By (\ref{chi12to3}) and (\ref{chi2to4}), if $$\chi(\fs_h^{\gamma}(n))\psi(c\fs_h(n)^{1-\gamma})>0$$ for some $0\leq h\leq k_0$, then
$$\chi^*(n^{\gamma})
\psi^*(cn^{1-\gamma})=1.
$$
So
$$
\sum_{h=0}^{k_0}\varpi(\fs_{h}(n))\chi(\fs_h^{\gamma}(n))\psi(c\fs_h(n)^{1-\gamma})-m\log 3X>0,
$$
i.e., there exist $0\leq h_1<h_2<\ldots<h_{m+1}\leq k_0$ such that
$$
\varpi(\fs_{h_j}(n))\chi(\fs_{h_j}^{\gamma}(n))\psi(c\fs_{h_j}(n)^{1-\gamma})>0
$$
for each $1\leq j\leq m+1$. Thus $\fs_{h_1}(n),\ldots,\fs_{h_{m+1}}(n)$ are the expected Piatetski-Shapiro primes.

\begin{Rem}
Motivated  by Theorem \ref{main}, we may propose the following conjecture.
\begin{conjecture} For any positive integer $k_0$ and any non-integral $c>1$, there exist infinitely many $n$ such that
$$
[n^c],\ [(n+1)^c],\ \ldots,\ [(n+k_0)^c]
$$
are all primes.
\end{conjecture}
\end{Rem}

\section*{Acknowledgement}
\medskip

This work is supported by the National Natural Science Foundation of
China (Grant No. 11271249) and the Specialized Research Fund for the
Doctoral Program of Higher Education (No. 20120073110059).

\end{document}